\documentclass[11pt]{amsart}
\usepackage{tikz}
\usepackage[margin=1.44in]{geometry}
\usepackage{color,soul,amssymb}
 
\title[Distinguishing simple groups]
{Distinguishing simple groups}
 
\author{Mariusz Grech, Andrzej Kisielewicz}
\address{Faculty of Pure and Applied Mathematics, Wrocław University of Science and Technology \\
Wybrzeże Wyspiańskiego Str. 27,
50-370 Wrocław, Poland}
\email{[mariusz.grech,andrzej.kisielewicz]@pwr.edu.pl}
\thanks{
{Supported in part by Polish NCN grant 2016/21/B/ST1/03079}}
 
 
\begin{document}

\newtheorem{Theorem}{Theorem}[section]
\newtheorem{Lemma}[Theorem]{Lemma}
\newtheorem{Fact}[Theorem]{Fact}
\newtheorem{Proposition}[Theorem]{Proposition}
\newtheorem{Corollary}[Theorem]{Corollary}

\newcommand{\A}{G}
\newcommand{\B}{H}
\newcommand{\C}{K}
\newcommand{\XX}{X}
\newcommand{\YY}{Y}
\newcommand{\xx}{x}
 
\renewcommand{\to}{\rightarrow}
\renewcommand{\gg}{g}
\newcommand{\hh}{h}

\renewcommand{\iff}{\longleftrightarrow}

 
\begin{abstract}
The distinguishing number $D(\Gamma)$ of a graph $\Gamma$ is the least size of a partition of the vertices of $\Gamma$ such that no non-trivial automorphism of $\Gamma$ preserves this partition. We show that if the automorphism group of a graph $\Gamma$ is simple, than $D(\Gamma)=2$. This is obtained by establishing the distinguishing number for all possible actions of simple groups. 
\end{abstract}
 
\maketitle

\section{Introduction}
 
The concept of distinguishing number of a graph has been introduced by Albertson and Collins \cite{AC} in 1996. Since then an extensive literature on this number and related topics has been developed and the area is still flourishing.
 
In 2004, Tymoczko \cite{tym} has extended the concept to the distinguishing number of an arbitrary group action (cf. \cite{cha,cha2}). This turned out to be very fruitful once, in 2011 \cite{BC}, it was realized that in permutation group theory the problem had been investigated for many years as a part of the study of regular sets. 
 

Given a group $\A$ acting on a set $\XX$, the \emph{distinguishing number} of this action, denoted $D_\XX(\A)$ or simply $D(\A)$, is the least size of a partition of $\XX$ such that no element of $\A$ preserves this partition, unless it fixes all $\xx\in \XX$. The distinguishing number of a graph $\Gamma$ is the distinguishing number of the automorphism group of $\Gamma$. 
Often, rather than about partition, authors speak about labeling (coloring) points (vertices) to destroy or break down the symmetry. Similarly, using the automorphism group, one defines the distinguishing number for other structures.
 
If the action of $\A$ on $\XX$ if faithful we speak of the distinguishing number of the (corresponding) permutation group $(\A,\XX)$. To have $D(\A)=2$ for such a group is equivalent to have a \emph{regular set} (that is, a subset of $\XX$ whose set-stabilizer in $\A$ is trivial \cite{CNS,SV}).

 
The oldest remarkable result in the area is that by Gluck \cite{glu} who proved in 1983 that 
if the order of an abstract group $\A$ is odd, then for each action $\A$ has a regular set, that is, $D(\A)=2$. This is also true for all primitive solvable permutation groups, but a few exceptions. Cameron et al. \cite{CNS,ser} showed that all but finitely many primitive permutation groups $\A$, not containing $A_n$, have $D(\A)=2$. Recently, similar results have been obtained for quasiprimitive and semiprimitive permutation groups in \cite{dev}. 
 
The Motion Lemma of Russell and Sundaram  \cite{rs} states that if the minimal degree of a permutation group $\A$ (the motion of the action) is at least $2 \log_2 |\A|$, then its distinguishing number is two. Using this, Conder and Tucker  cite{CT} in 2011 proved  that if we consider a vector space, group, or map, then in all but finitely many cases, the action of the automorphism group has the distinguishing number two. They also suggested that, in general, having the distinguishing number two is a generic property. 
 
Similar results, confirming this belief,  have been proved for general linear groups in \cite{cha,KWZ}. Further, if one takes a Cartesian product of enough copies of the same graph, then the distinguishing number is two (see \cite{alb,IK}). Also, for all maps with more than 10 vertices, the action of the automorphism group on the vertices has distinguishing number two (see  \cite{tuc,tuc2}, cf. \cite{PT}).

Now, it is important to realize that even if two graphs have the same automorphism group their distinguishing numbers may be different. This is so, because this number depends on the action of a group, not merely on its abstract structure. Many exceptions to the 
generic case  of 2-distinguishability  are connected with intransitive actions. The study of distinguishing number involves necessarily intransitive graphs. Therefore, as we demonstrate in this paper, the details of the construction of intransitive permutation groups may be especially useful in this study. 
 
The distinguishing number is a local property of graphs in the sense that it is enough to replace one vertex by a clique to increase this number. This leads to a Brooks type inequality $D(\Gamma) \leq \Delta(\Gamma)+1$ for connected graphs (see \cite{CTr,KWZ}). We believe however, on the basis of our study,  that a real obstacle here are the maximal sizes $\omega(\Gamma)$ of a clique and $\omega(\overline \Gamma)$ of an anticlique, and in contrast with the chromatic number of a graph this provides an actual upper bound for the distinguishing number.
\bigskip
 
\noindent\textbf{Conjecture.}\textit{
Let $\Gamma$ be an arbitrary graph. Then
$$ D(\Gamma) \leq 1+\max\{\omega(\Gamma),\omega(\overline \Gamma)\}.$$
}
 
The complete bipartite graph $K_{n,n}$ shows that this inequality cannot be sharpened. 
 
On the other hand, for intransitive graphs, it is easy to see that $D(\Gamma)$ is bounded from above by the maximum of distinguishing numbers of the action of the automorphism group on the orbits. Yet, this bound is very rough. We show that in case of simple group the maximum may be replaced by the minimum.  Constructions of intransitive groups, corresponding to the edges joining vertices in different orbits of graphs, usually decrease the distinguishing number in a radical way. 
 
In this paper we deal with the case when the automorphism group $Aut(\Gamma)$ of a graph $\Gamma$ is simple. We show that this case is instructive in many aspects. It may be viewed as the other extreme of solvable groups, and it is natural to check if having distinguishing number two is also typical in this very case. The problem is that, while we know that the class of graphs with simple automorphism group is large and diverse, we have no general results on such graphs that could be useful to compute their distinguishing number.
Therefore, our approach, rather than considering graphs, is to consider all possible actions of simple groups, and exclude those that are certainly not the automorphism groups of graphs. 
 
Of course, we make use of that we have extensive and  deep knowledge of simple groups. From the results established for primitive and quasiprimitive permutation groups we can infer that, with known exceptions, transitive actions of simple groups have the distinguishing number two. So, in this paper, we focus on intransitive actions. Using a general result on the structure of intransitive permutation groups we show that in case of simple groups these actions have a special and very nice structure of parallel sums. This allows to check effectively all possible exceptions, using known classification results and computation machinery.

As we shall see, the case of simple groups illustrates perfectly all the discussed phenomena and may be a good starting point in further study. A counterpart of a special role of cliques and anticliques for graphs is the role of the alternating groups among simple groups. Our paper shades some light on this. 
 
For computer calculations we have used system GAP 4.10.1. In cases of applying GAP, we give only those computational details that seem sufficient for the reader to verify our results using any computational system for permutation groups. Since in the case of simple groups all actions are faithful, we use the terminology of permutation groups rather than those of actions of groups.

\section{Intransitive simple groups}\label{s:ss}
 
To describe the structure of intransitive simple groups we need to recall the concept of the \emph{subdirect sum of permutation groups} and the general result on the structure of intransitive  permutation groups.  We give also some comments that allow to apply the recalled results in a proper and easy way.
All considered permutation groups are finite and considered up to \emph{permutation isomorphism} \cite[p.~17]{DM} (i.e., two groups that differ only in labeling of points are treated as the same).
 
Given two permutation groups $\A \leq Sym(\XX)$ and $\B \leq Sym(\YY)$, the \emph{direct sum} $\A\oplus \B$ is the permutation group on the \emph{disjoint}
union $\XX \cup \YY$ defined as the set of all permutations $(\gg,\hh)$, $\gg\in \A, \hh\in \B$ such that (writing permutations on the right)
$$
\xx(\gg,\hh) =
\left\{\begin{array}{ll}
\xx\gg, & \mbox{if } \xx\in \XX\\
\xx\hh, & \mbox{if } \xx\in \YY
\end{array}\right.
$$
Thus, in $\A\oplus \B$, permutations of $\A$ and $\B$ act independently in a natural way on the disjoint union of the underlying sets.
 
We define the notion of the \emph{subdirect sum} following \cite{GK1} (and the notion of \emph{intransitive product} in \cite{kis1}).
Let $\B_1\lhd\; \A_1 \leq S_n$ and $\B_2\lhd\; \A_2\leq S_m$ be permutation groups such that $\B_1$ and $\B_2$ are normal subgroups of $\A_1$ and $\A_2$, respectively. Suppose, in addition, that factor groups $\A_1/\B_1$ and $\A_2/\B_2$ are (abstractly) isomorphic and $\phi : \A_1/\B_1 \to \A_2/\B_2$ is the isomorphism mapping. Then, by
$$
\A = \A_1[\B_1] \oplus_\phi \A_2[\B_2]
$$
we denote the subgroup of $\A_1 \oplus \A_2$ consisting of all permutations $(\gg,\hh)$, $\gg\in \A_1, \hh\in \A_2$, such that $\phi(\B_1\gg) = \B_2\hh$.
Each such group will be called the \emph{subdirect sum} of $\A_1$ and $\A_2$.
 
If $\B_1=\A_1$ and $\B_2=\A_2$, then $\A = \A_1 \oplus \A_2$ is the usual direct sum of $\A_1$ and $\A_2$.
If $\B_1$ and $\B_2$ are trivial one-element subgroups, $\phi$ is the isomorphism of $\A_1$ onto $\A_2$, and the sum is called, in such a case, the \emph{parallel sum} of $\A_1$ and $\A_2$. Then the elements of $\A$ are of the form $(\gg,\phi(\gg))$, $\gg\in \A_1$, and both the groups act in a parallel manner on their sets \emph{via} isomorphism $\phi$. In this case we use the notation $\A=\A_1||_\phi \A_2$, where $\phi$ is an (abstract) isomorphism between $\A_1$ and $\A_2$, or simply $\A=\A_1|| \A_2$ if there is no need to refer to $\phi$.
Note that $\A_1$ and $\A_2$ need to be abstractly isomorphic, but not necessarily permutation isomorphic, and they may act on sets of different cardinalities.
 
In the special case when, in addition, $\A_1=\A_2=\A$ and $\phi$ is the identity, we write $\A^{(2)}$ for $\A||_\phi \A$. More generally, for $r\geq 2$, by $\A^{(r)}$ we denote the permutation group in which the group $\A$ acts in the parallel way (\emph{via} the identity isomorphisms) on $r$ disjoint copies of a set $\XX$. This group is called
the \emph{parallel multiple} of $\A$, and its element are denoted $\gg^{(r)}$ with $\gg\in \A$. In particular, we admit $r=1$ and put $\A^{(1)}=\A$.
For example, the cyclic group generated by the permutation
$\gg = (1,2,3)(4,5,6)(7,8,9)$ is permutation isomorphic to the parallel multiple $C_3^{(3)}$, where $C_3$ denotes the permutation group on $\{1,2,3\}$ generated by the cycle $(1,2,3)$.
 
The main fact established in \cite{kis1} is that every intransitive group has the form of a subdirect sum, and its components can be easily described. Let $\A$ be an intransitive group acting on a set $\XX = \XX_1 \cup \XX_2$ in such a way that $\XX_1$ and $\XX_2$ are disjoint fixed blocks of $\A$. Let $\A_1$ and $\A_2$ be restrictions of $\A$ to the sets $\XX_1$ and $\XX_2$, respectively (they are called also \emph{constituents}). Let $\B_1 \leq \A_1$ and $\B_2 \leq \A_2$ be the subgroups fixing pointwise $\XX_2$ and $\XX_1$, respectively. Then we have
 
\begin{Theorem} \cite[Theorem 4.1]{kis1} \label{th:s}
If $\A$ is a permutation group as described above, then
$\B_1$ and $\B_2$ are normal subgroups of $\A_1$ and $\A_2$, respectively,
the factor groups $\A_1/\B_1$ and $\A_2/\B_2$ are abstractly isomorphic, and
$$\A = \A_1[\B_1] \oplus_\phi \A_2[\B_2],$$
where $\phi$ is an isomorphism of the factor groups.
\end{Theorem}

There is some subtlety regarding the parallel powers we have to be aware. A well-known fact is that the automorphism group $Aut(\A)$ of a group $\A$ may have outer automorphisms that are not given by the conjugation action of an element of $\A$. In the case of permutation groups $\A\leq Sym(\XX)$ some outer automorphism may still be given by the conjugation action of an element in $Sym(\XX)$. This corresponds generally to permuting elements of $\XX$, and such automorphisms are called \emph{permutation automorphisms}. They form a subgroup of $Aut(\A)$, which we denote by $PAut(\A)$. Often $PAut(\A)=Aut(\A)$, but some permutation groups have also other automorphisms, which we will call \emph{nonpermutation} automorphisms.
 
Now, if $\A = \B\oplus_\psi \B$, for some permutation group $\B$, where $\psi\in PAut(\A)$, then $\A$ is permutation isomorphic to $\B^{(2)}$ (i.e., $\A=\B^{(2)}$, according to our convention). If $\psi$ is a nonpermutation automorphism, then $\A \neq \B^{(2)}$. (More precisely, we should speak here about isomorphisms induced by automorphisms and make distinction between base sets of components, but we assume that this is contained in the notion of the \emph{disjoint union}, and we will make it explicit only when the need arises). An example is the alternating group $A_6$ that has a nonpermutation automorphism $\psi$ (cf. \cite{CL}). Then, $A_6^{(2)}$ and $A_6 ||_\psi A_6$ are not permutation isomorphic. As we shall see, this example is exceptional. The distinguishing number $D(A_6 ||_\psi A_6)=3$. 
 
 
With this result we may turn to describing the structure of intransitive simple permutation groups.
Without loss of generality we may assume that groups we consider have no fixed points, as fixed points do not affect the distinguishing number. 
More precisely, if a permutation group $\A$ has precisely $m>0$ fixed points, then $\A=\A'\oplus I_m$, where $\A'$ has no fixed points, and $I_m$ is the \emph{trivial} group acting on $m$ points, that is, containing only the identity permutation. Obviously, $D(\A)=D(\A')$, and $\A$ and $\A'$ are abstractly isomorphic. We have the following.

\begin{Theorem}\label{p:simple}
Let $\A$ be a simple intransitive group with no fixed points. Then $A$ has the form of a parallel sum $\A=\B||_\phi \C$, where $\B$ and $\C$ are permutation groups abstractly isomorphic to $\A$.
\end{Theorem}
\begin{proof}
By Theorem~\ref{th:s},  $\A = \B[\B']\oplus_\phi \C[\C']$, where each of $\B$ and $\C$ acts nontrivially on at least two points. Now, the group $B'\oplus I_m$ is a normal subgroup of $\A$, as $B'\lhd B$. Since $C$ is nontrivial, and $\A$ is simple, we infer that $\B'$ is trivial. Similarly, we observe that $\C'$ is trivial. It follows that $A=\B||_\phi \C$, 
$\phi$ is an isomorphism between $\B$ and $\C$, and $\A$ is isomorphic to both $\B$ and $\C$, as required. 
\end{proof}
 
The proposition above means that intransitive simple permutation groups with nontrivial orbits have always the form of parallel sums. Some remarks are needed to make a proper use of this result.
 
Note that (using the inverse isomorphism $\phi^{-1}$) we see easily that $G_1||_\phi G_2$ and $G_2 ||_{\phi^{-1}} G_1$ are permutation isomorphic, and since we treat permutation isomorphic groups as identical, the operation of the parallel sum may be considered to be commutative. Further, decomposing each summand step by step we can get a decomposition into transitive components. In particular,
in $G=(G_1 ||_\phi G_2)||_\psi G_3$ all the involved groups must be abstractly isomorphic, and $G$ is permutation isomorphic with $G_1 ||_{\phi'} (G_2||_{\psi'} G_3)$, where isomorphisms $\phi'$ and $\psi'$ are suitably determined by $\phi$ and $\psi$. Thus we may also consider this operation to be associative (up to permutation isomorphism).
 
So, generally, a simple intransitive permutation group (with no fixed points) is a parallel sum of two or more transitive components that are all abstractly isomorphic, and the action on the union of orbits is given by a system of suitable isomorphisms between components. 
 
Note however, that in case, when a group has two different actions on the set of the same cardinality, various systems of isomorphism may lead to different permutation groups; so the pointing out only transitive components may not define the parallel sum uniquely. (For example, projective symplectic group $PSp(4,3)$ has two nonequivalent actions on the set of cardinality $n=40$). A similar remark applies when there are two identical components with a nonpermutation automorphism.

\section{Distinguishing number and regular sets}\label{s:is}

Consider distinguishing labelings for parallel sums. 
It is easy to see that in this case, if we have a distinguishing $k$-labeling (i.e., one with $k$ labels) for one of the components, then
it is easy to construct a distinguishing $k$-labeling for the whole sum.
 
\begin{Lemma}\label{l:reg}
Let $\A= \B ||_\phi \C$ be a parallel sum of two groups. Then,  $$D(\A) \leq \min\{D(\B),D(\C)\}.$$
\end{Lemma}
\begin{proof} 
Suppose that in the definition of $\A$ above,  $\B$ acts on $\XX$ and $\C$ act on $\YY$. Note that, by the definition of the parallel sum, the only permutation in $\A$ fixing all the points in $\XX$ is the identity. Therefore, given  is a having a distinguishing $k$-labeling $\ell_1$ of $\XX$ for the action of $\B$, one can obtain a distinguishing $k$-labeling $\ell$ of $\XX\cup\YY$ for the action of $\A$ just by assigning an arbitrary label to all points in $\YY$. A symmetrical arguments proves the claim.
\end{proof}

Now, let us consider the parallel multiple $\A=A_n^{(k)}$ of the alternating group $A_n$. As we will see below, if $k$ is large enough 
with regard to $n$, then $D(\A)=2$, which is one more confirmation of the phenomenon discussed in the introduction.

\begin{Lemma}\label{l:An}
If $\A= A_n^{(k)}$ is the parallel multiple of the alternating group $A_n$, $n\geq 3$ and $k\geq 1$, then 
$D(\A)$ is the smallest integer $d$ such that $d^k\geq n-1$. 
\end{Lemma}
\begin{proof}  
We need to show that for $d^k\geq n-1$, $\A$ has a distinguishing $d$-labeling, and if $d^k < n-1$, then no such labeling exists for $\A$. 
 
We may assume that $A_n^{(k)}$ acts on $\XX = \XX_1 \cup \ldots \cup \XX_{k}$, where each $\XX_i$ is a copy of $\{1,2,\ldots,n\}$, and the action of $g^{(k)}$ is the same on all copies as the action of $g\in A_n$ on $\{1,\ldots,n\}$. Let  $\ell$ denote a $d$-labeling of $\XX$ with $\ell(i,j)$ denoting the label of $i$-th element in $\XX_j$. We define a labeling $\ell'$ of $\{1,2,\ldots,n\}$ by $\ell'(i) = (\ell(i,1),\ell(i,2),\ldots,\ell(i,k))$. Obviously, $\ell$ is a distinguishing $d$-labeling for $A_n^{(k)}$ if and only if $\ell'$ is distinguishing for $A_n$ on $\{1,2,\ldots,n\}$. The later holds if and only if no more than two points have the same label (otherwise there is a permutation in $A_n$ preserving the labeling). Since the labels  $\ell'(i)$ are $k$-tuples of $d$ possible values, there are $d^k$ of them. So, if $d^k < n-1$, then either more then two points must have the same label or there are two pairs with the same label, and therefore the labeling is not distinguishing. Otherwise, One may choose labels $\ell(i,j)$ so that $\ell$ is distinguishing. This proves the result.
\end{proof}
 
In general, if $\A$ is a parallel sum of permutation isomorphic components, it needs not to be a parallel multiple. As we have already explained, this is so, since permutation groups may have nonpermutation automorphisms. 
For alternating groups we have an interesting exception mentioned in the previous section.

\begin{Lemma}\label{l:Ann}
Let $\A$ be the parallel sum of components permutation isomorphic to a fixed alternating group $A_n$, such that $\A\neq A_n^{(r)}$ for any $r\geq 1$. Then, $n=6$, $\A= A_6^{(r)}||_\psi A_6^{(s)}$ for some $r,s\geq 1$ and $D(\A)=2$ with the exception of $\A= A_6||_\psi A_6 \neq A_6^{(2)}$, in which case $D(A_6||_\psi A_6)=3$
\end{Lemma}
 
\begin{proof}
It is well known \cite{CL} that the only alternating group $A_n$ having a nonpermutation automorphism is $A_6$ and that the index $[Aut(A_6) : PAut(A_6)]=2$, which means that up to permutation automorphisms there is only one nonpermutation automorphism in $A_6$. Hence, if $\A \neq A_n^{(k)}$, then $ \A=A_6^{(r)}||_\psi A_6^{(s)}$.
 
First, consider the case of two components $\A= A_6 ||_\psi A_6$.
Using GAP, we find a nonpermutation automorphisms $\psi$ of $A_6$ given by the images of generators of $A_6$:
$$\big( (2,3)(4,5) \big)\psi = (2,5)(3,4), \hbox{\rm\ and \ } \big( (1,2,3,4)(5,6)\big)\psi = (1,2,3,4)(5,6). $$
(Note that this mapping cannot be obtained by permuting points). This may be used to form the group $G=A_6||_\psi A_6$ as one generated by
$$g= (2,3)(4,5)(2',5')(13',14')
\hbox{\rm\ and }h=(1,2,3,4)(5,6)(1',2',3',4')(5',6')$$
on the set $\XX=\XX_1\cup\XX_2$ with $\XX_1 =\{1,\ldots,6\}$ and $\XX_2=\{1',\ldots,6'\}$.

One can check, still using GAP, that $A_6 ||_\psi A_6$ has no regular set, and  $D(A_6 ||_\psi A_6)=3$. In turn, the group $\A' = A_6^{(2)} ||_\psi A_6$, obtained by applying the same nonpermutation automorphism, has regular sets of all sizes from $4$ to $14$. It follows that $D(\A')=2$, and using Lemma~\ref{l:reg}, $D(A_6^{(r)}||_\psi A_6^{(s)})=2$ for all $r+s>2$.  
\end{proof}

Now, we will combine the above lemmas with what is known on primitive groups. In \cite{ser}, Seress lists all primitive groups that have no regular set (cf. \cite[Theorem~2.2]{SV}). Using this list one can distinguish all simple groups in primitive action that have no regular set. In the list below the first entry denotes the degree of the action, and the second---the group itself. If the group is abstractly isomorphic to one of $A_n$ or another group in the list with a different name, then this is indicated. There are no other isomorphisms between the groups in the list except those indicated. (This may be inferred from the classification of the finite simple groups, using, e.g., \cite[Appendix~A]{DM}). Note that all groups in the list appears in their natural actions with the exception of $L_2(11)$ which acts here on $11$ points (rather than $12$).
\bigskip
 
${\mathcal L} = \{(6, L_2(5)\!\cong\! A_5)$,
$(7, L_3(2)\!\cong\! L_2(7))$,
$(8, L_2(7)\!\cong\! L_3(2))$,
$(9, L_2(8))$,
 
\hspace*{8mm}$(10, L_2(9)\!\cong\! A_6)$,
$(11, L_2(11))$,
$(11, M_{11})$,
$(12, M_{11})$,
$(12, M_{12})$,

\hspace*{8mm}$(13, L_3(3))$,
$(15, L_4(2) \!\cong\! A_8)$,
$(22, M_{22})$,
$(23, M_{23})$,
$(24, M_{24})\}$.
\bigskip
 
Performing suitable computation one may check that although these groups have no regular sets, their sums have. First we establish the following.

\begin{Lemma}\label{l:h2}
For each group $\B$ in the list $\mathcal L$,
$\B^{(2)}$ has a regular set.
\end{Lemma}

\begin{proof}
This can be easily checked using GAP or other systems for computation in permutation groups. 
To provide easy verification for the reader we give some details of computations.
Let $n$ be the degree of $\B$, $\XX=\{1,2,\ldots,n\}$, and $\YY= \{1',2',\ldots,n'\}$. 
 
As an example we take the group
$\B=L_2(5)$ acting on $n=6$ elements. To point out a regular set for $\B^{(2)}$  we need to define the underlying set $\XX$ and generating permutations for $\B^{(2)}$. In each case for $\XX$ we take $\XX=\{1,2,\dots,n\} \cup\{1',2',\ldots,n'\}$. Then, we find a set of generators for $\B$ (one may use, for example, the GAP listing of primitive groups of degree $n$). In the considered example we get
$\gg = (1,2,5)(3,4,6)$ and $\hh =(3,5)(4,6)$.
We form permutations $\gg^{(2)}$ and $\hh^{(2)}$ obtaining
\begin{align*}
\gg^{(2)} &= (1,2,5)(3,4,6)(1',2',5')(3',4','6), \\
\hh^{(2)} &= (3,5)(4,6)(3',5')(4',6').
\end{align*}
Clearly, the above permutations generate $\B^{(2)}$. Now the reader can check that, for instance, the set $S=\{1,2,3,2',4',6'\}$ has the trivial stabilizer.
 
 
\begin{table}[ht!]
\begin{center}
\caption{}
\label{t:1}
 
\small
\begin{tabular}{ |l|l|l|l| }
\hline
$n$ &$\B$ & generators $\gg$ and $\hh$ of $\B$ & regular set $S$ in $\B^{(2)}$ \\ \hline
$6$ & $L_2(5)$ & {$(1,2,5)(3,4,6),$} & $\{1,2,3,2',4',6'\}$ \\
& & {$(3,5)(4,6)$} & \\ \hline
$7$ & $L_3(2)$ & {$(1,4)(6,7),$} & $\{1,2,3,3',5'\}$ \\
& & {$ (1,3,2)(4,7,5)$} & \\ \hline
$8$ & $L_2(7)$ & {$ (3,7,8)(4,6,5),$} & $\{1,2,2',7'\}$ \\ & & {$ (1,4,2,5,7,8,6)$} & \\ \hline
$9$ & $L_2(8)$ & {$ (1,5,4,2,8,3,6),$} & $\{1,2,3,4,2',3',4',5'\}$ \\ & & {$ (1,8,6,2,7,3,9)$} & \\ \hline
$10$ & $L_2(9)$ & {$ (1,9,6,3,8)(2,10,7,5,4),$} & $\{1,2,3',5'\}$\\
& & {$ (1,10,6,2,5)(3,4,9,8,7)$} & \\ \hline
$11$ & $L_2(11)$ & {$ (1,5)(2,4)(3,10)(7,11),$} & $\{1,2,3',5',6'\}$\\
& & {$ (3,11,5)(4,7,9)(6,8,10)$} & \\ \hline
$11$ & $M_{11}$ & {$ (1,2,3,4,5,6,7,8,9,10,11)
,$} &
$\{1,2,3,1',5',7'\}$\\
& & {$
(3,7,11,8)(4,10,5,6)
$} & \\ \hline
 
$12$ & $M_{11}$ & {$ (1,12)(2,10,5,7)(3,8)(4,6,11,9)
,$} &
$\{1,2,3,4,1',4',7'\}$\\
& & {$
(1,3)(2,7)(8,11)(9,10)
$} & \\ \hline
$12$ & $M_{12}$ & {$ (1,4,12,6)(2,7,5,9,8,10,3,11),$} &
$\{1,2,3,4,11,1',2',5',$\\
& & {$
(1,12)(2,6,4,9,7,8,11,3)
$} & $8',9'\}$ \\ 
\hline
$13$ & $L_3(3)$ & {$ (1,10,4)(6,9,7)(8,12,13),$} &
$\{1,2,3,4,1',7',9'\}$\\
& & {$
(1,3,2)(4,9,5)(7,8,12)(10,13,11)
$} & \\ 
\hline
 
$15$ & $L_4(2)$ & {$ (1,9,5,14,13,2,6)(3,15,4,7,8,12,11),$} &
$\{1,2,3,4,5,6,9,1',$\\
& & {$
(1,3,2)(4,8,12)(5,11,14)(6,9,15)
$} & $7',8',11',15'\}$ \\ 
\hline
 
$22$ & $M_{22}$ & {$ (1,13,11,17)(2,7)(3,22,12,21)(4,18,16,10)$} & $\{1,2,3,4,7,9,1',3',8'\}$\\ 
& & {$(6,20,19,14)(9,15),$} & \\ 
& & {$ (1,6,12,11,14,5,22)(2,19,16,9,13,21,8)$} & \\
& & {$ (3,17,18,15,7,4,10)$} & \\ \hline
 
$23$ & $M_{23}$ & {$ (1,2,3,4,5,6,7,8,9,10,11,12,13,14,15,16,$} & $\{1,2,3,4,9,12,1',10',$\\ 
& & {$ 17,18,19,20,21,22,23),$} & $17'\}$\\ 
& & {$ (3,17,10,7,9)(4,13,14,19,5)(8,18,11,12,23)$} & \\
& & {$ (15,20,22,21,16)$} & \\ \hline
 
$24$ & $M_{24}$ & {$ (1,5)(2,14,7,12)(3,21)(4,17,16,11)$} & $\{1,2,3,5,11,13,19,1'$\\ 
& & {$ (6,20,23,22)(9,10,15,13),$} & $4',14',18',24'\}$\\ 
& & {$ (1,19,15,8,20,23,24,9,14,11,5,10,22,13,2)$} & \\
& & {$ (3,6,4)(7,16,12,17,18)$} & \\ \hline
\end{tabular}
\end{center}
\end{table}

The proof in other cases is the same. The only change is in the choice of permutations $\gg$ and $\hh$ generating $\B$, and the regular set $y$. These details are given in Table~\ref{t:1}. \end{proof}

We also need to consider subdirect sums of elements from $\mathcal L$ that arise by using a nonpermutation automorphism. 
It is easy to check that the groups in $\mathcal L$ having nonpermutation automorphisms are the following:
$${\mathcal L}^* =\{ L_3(2), L_2(11), M_{12}, L_3(3), L_4(2) \}.$$
Moreover, one may also check that the index $[Aut(G):PAut(G)]=2$ in each of these cases. We use this to prove the following.

\begin{Lemma}\label{l:hh}
For each group $\B$ in the list $\mathcal L$, and each automorphism $\psi$  of $\B$, 
if the parallel sum $\A=\B||_\psi \B$ is different from  $\B^{(2)}$, then $\A$ has a regular set.
\end{Lemma}

\begin{proof}
In view of the remarks above we may restrict to groups $\B\in {\mathcal L}^*$ and to one nonpermutation automorphism $\psi$ in each case. 
 
As in the proof of Lemma~\ref{l:Ann}, first we find a nonpermutation automorphism of $\B$. For $\B=L_3(2)$, using GAP, one can find a nonpermutation
automorphism $\psi$ given by the images of generators 
$$\big( (1,2)(5,7)\big)\psi = (1, 2)(3, 6), \hbox{\rm\ and \ } \big( (2,3,4,7)(5,6)\big)\psi =
(2, 3, 4, 7)(5, 6).$$
Hence, the group $\A=\B||_\psi \B$ on the set $\XX=\{1,\ldots,7\} \cup \{1',\ldots, 7'\}$ is generated by permutations
$$ \gg =(1,2)(5,7)(1',2')(3',6') \hbox{\rm\ and }
\hh= (2,3,4,7)(5,6)(2',3',4',7')(5',6').$$

It is easy to check that $\B$ has regular sets of all sizes from $4$ to $10$, and one of them is, for example, $\{1,5,2',3'\}$.

For other groups in ${\mathcal L}^*$ the constructions are the same, and the only difference is the nonpermutation automorphism used. In Table~\ref{t:1b}  generators of $H||_\psi H$ and examples of regular sets are given. Generators are given in the form as above, allowing to decode nonpermutation automorphisms used.
\end{proof}
 
\begin{table}[hb!]
\begin{center}
\caption{}
\label{t:1b}
 
\small
\begin{tabular}{ |l|l|l|l| }
\hline
$n$ &$\B$ & generators $\gg$ and $\hh$ of $\B||_\psi\B$ & regular set  \\ \hline
$7$  & $L_3(2)$ & {$(1,2)(5,7)(1',2')(3',6'),$} & $\{1,5,2',3'\}$ \\
& & {$(2,3,4,7)(5,6)(2',3',4',7')(5',6')$} & \\ \hline
 
$11$ & $L_2(11)$ & {$(1, 3)(2, 7)(5, 9)(6, 11)
(1', 4')(2', 3')(5', 10')(9', 11'),$} & $\{1,2,3,3',5'\}$ \\
& & $ (3, 5, 11)(4, 9, 7)(6, 10, 8)(3', 5', 11')(4', 9', 7')(6', 10', 8')$ & \\ \hline

$12$ & $M_{12}$ & $ (1,3,5,7,2,4,8,9)(6,10,11,12) (1',3')(5',4',6',7',8',$ & $\{1,2,3,7,9,1',3',$\\
&& $10',9',12'),$ & $\;\; 5',9',12'\}$\\  & & $ (1,7,8,3,9,6,4,11,10,12,2)
(3',2',1',7',11',9',12',5',$ & \\
&& $ 6',4',10')$ & \\ \hline
 
$13$   & $L_3(3)$ & $ (3,5,11)(6,7,9)(8,12,13)
(3', 8', 7')(5', 12', 9')(6', 11', $ & $\{1,2,3,1',2',3'\}$\\
& & $13'),$ & \\
 & & $ (1,13,7)(2,10,6)(3,5,12)(4,11,9)
(1', 13', 7')(2', 10', 6')$&\\
&&$ (3',5', 12')(4', 11', 9')$ & \\ \hline

$15$   & $L_4(2)$ & $ (1, 9, 5, 14, 13, 2, 6)(3, 15, 4, 7, 8, 12, 11)
(1', 4', 2', 14', 13',$& $\{1,3,5,3',9',13'\}$\\
&&$7', 8')(3', 10', 15', 9', 5', 6', 12'),$  &\\ & & $ (1, 3, 2)(4, 8, 12)(5, 11, 14)(6, 9, 15)(7, 10, 13)
(1', 2', 3')$&\\
&&$(4', 14', 10')(5', 12', 9')(6', 13', 11')(7', 15', 8')$ & \\ \hline
 
\end{tabular}
\end{center}
\end{table}

In the next lemma we consider the last case needed to prove our main result.
 
\begin{Lemma}\label{l:pary}
For each pair of groups $H \cong K$ in the list $\mathcal L$ that are different as permutation groups, the group $H ||_\phi K$ has a regular set.
\end{Lemma}
 
\begin{proof}
Looking at the list $\mathcal L$
we see that there are exactly five possibilities for $H$ and~$K$:
 
\renewcommand{\labelenumi}{(\roman{enumi})}
\begin{enumerate}
\item $n=5, 6$ with $A_5 \cong L_2(5)$,
\item {$n=7, 8$ with $L_3(2) \cong L_2(7)$,}
\item $n=6, 10$ with $A_6 \cong L_2(9)$,
\item {$n=11, 12$ with two actions of $M_{11}$,}
\item $n=8, 15$ with $A_8 \cong L_4(2)$,
\end{enumerate}
 
Again, for each pair calculations are similar, so we comment only one example.
 
As the generators of $\A=L_2(5)||_\phi A_5$ (formed similarly as the proof of the previous lemma) we take $ (1,3,4)(2,5,6)(8,9,11)$ and
$(1,2)(3,4)(7,8)(9,10)$.
 
We need to note that the group $\A=L_2(5)||_\phi A_5$ is unique up to permutation isomorphism. This is because $A_5$ has only permutation automorphism (even if $L_2(5)$ has a nonpermutation automorphism). The situation is similar in the remaining cases (one of the groups has only permutation automorphisms), so we have only one permutation group of the form $\B||_\phi \C$ in each case. 
 
Using GAP, it is not difficult to find a regular set in $L_2(5)||A_5$ (in this case we found that $\{1,3,7,9\}$ is regular).
 
In Table~\ref{t:2}, we list the degree $n$ in the form the sum of the degrees of transitive components $\B$ and $\C$, the generators of the sum $\B||\C$, and an example of a regular set in $\B||\C$. This allows to decode the isomorphism $\psi$ and to check easily that the set pointed out is regular.
\end{proof}
 
 
\begin{table}[ht]
\begin{center}
\caption{}
\label{t:2}
 
\small
\begin{tabular}{ |l|l|l|l|}
\hline
$n$ &group & generators of $H||_\phi  K$ & regular set \\ \hline
 
 $6+5$ & $L_2(5)||_\phi A_5$ &  (1,3,4)(2,5,6)(8,9,11), & \{1,3,7,9\} \\
&& (1,2)(3,4)(7,8)(9,10) &\\ \hline
 
$8+7$& $L_2(7)||_\phi  L_3(2)$ &
(1,6,5)(2,3,7)(9,11,10)(12,15,13, & 
\{1,3,5,10,12,14\} \\
&& (1,4)(2,7)(3,5)(6,8)(9,12)(14,15) &\\ \hline
 
$10+6$&$L_2(9)||_\phi A_6$ & (1,5,3,9,6)(2,7,8,4,10)(11,12,13,14,15), & \{1,3,13,15\} \\
&& (1,5,3,9,6)(2,7,8,4,10)(11,13,14,15,16) & \\ \hline
 
$12+11$&$M_{11}(12)||_\phi M_{11}$ & (1,12)(2,10,5,7)(3,8)(4,6,11,9)(13,21,
&  \{1,3,7,9,11,14,18\} \\
 
&&
17,19)(16,20,23,18), &  \\
 
&& (1,3)(2,7)(8,11)(9,10)(14,16)(15,18)(19, & \\
 
&& 22)(21,23)& \\ \hline

$15+8$ & $L_4(2)||_\phi A_8$ & (1,9,5,14,13,2,6)(3,15,4,7,8,12,11)(16, & \{1,4,7,8,18,20,23\}\\
 
&& 17,18,19,20,21,22), &  \\
&& (1,3,2)(4,8,12)(5,11,14)(6,9,15)(7,10, & \\
&& 13)(21,22,23) & \\\hline
 
\end{tabular}
\end{center}
\end{table}

\section{Results}
 
Now, we are ready to prove our main result. Below the well-known notation for abstract groups is used for permutation groups obtained from the corresponding abstract group in its natural action.  An exception is $(L_2(11),11)$ meaning the  projective special linear group $L_2(11)$ acting on $11$ points. $A_6 ||_\psi A_6$ denotes the exceptional intransitive group described in Lemma~\ref{l:Ann}.

\begin{Theorem}
Let $\A$ be a simple permutation group with no fixed points. Then, $D(\A)=2$ except for the following cases:
\begin{enumerate}
\item If $\A=A_n^{(k)}$,  $n>4$, then $D(\A)$ is the smallest integer $d$ such that $d^k\geq n-1$;
\item If $\A \in \{L_3(2), M_{11}, M_{12}\}$, then $D(\A)=4$.
 
\item If $\A \in \{L_2(5),L_2(7),L_2(8),(A_6,10), (L_2(11), 11), (M_{11}, 12), L_3(3), (A_8, 15), $
 
$M_{22}, M_{23}, M_{24},
A_6 ||_\psi A_6\}$ then $D(\A)=3$.
\end{enumerate}
\end{Theorem}

\begin{proof}
If $\A$ is primitive, then by \cite{ser}, we get the list $\mathcal L$ of fourteen exceptional simple permutation groups that have no regular set. The distinguishing numbers for these groups are computed in \cite{dev}. If $\A$ is transitive imprimitive, then as a simple group is quasiprimitive, and by  \cite[Theorems~2]{dev}, $D(\A)=2$. 
 
Thus, we may assume that $\A$ is intransitive, and by Theorem~\ref{p:simple} (and the following remarks), $\A$ is a parallel sum of at least two transitive components, all abstractly isomorphic to $\A$. 
 
If at least one of the components $\C$ is different from $A_n$ and is not on the list $\mathcal L$, then $D(\C)=2$, and the result follows by Lemma~\ref{l:reg}. So, we may assume that all the components are isomorphic to some $A_n$ or one of the groups in the list $\mathcal L$.
 
Suppose first that $\A=\B^{(k)}$ is a parallel multiple. Then, in the case when $\B=A_n$, the result follows by Lemma~\ref{l:An}, and in case when $\B\in \mathcal L$, the result follows by Lemma~\ref{l:h2} combined with Lemma~\ref{l:reg}.  
 
Next, suppose that $\A$ is a parallel sum of the same components $\B$ different from $\B^{(k)}$. Then, if $\B=A_n$, the result follows by Lemma~\ref{l:Ann}. The fact that $D(A_6||_\psi A_6)=3$ may be easily checked using GAP. If $\B\in \mathcal L$, then the result is by Lemma~\ref{l:hh} and Lemma~\ref{l:reg}. 
 
It remains to consider the situation when $\A$ has two transitive components $\B$ and $\C$ that \emph{are not} permutation isomorphic, but are abstractly isomorphic. In this case the result follows by Lemma~\ref{l:pary} combined with Lemma~\ref{l:reg}.
\end{proof}
 
\begin{Corollary}\label{c:}
If the automorphism group of a graph $\Gamma$ is simple, then $D(\Gamma)=2$.
\end{Corollary}
\begin{proof}
First, if $\Gamma$ is transitive, then  the theorem above shows that all transitive exceptions are double-transitive, which means that none of them is the automorphism group of a graph.
 
For $\A=A_6 ||_\psi A_6$ we consider the orbitals of $\A$ (orbits in the action of $\A$ on pairs). There are three orbitals: two corresponding to each orbit consisting of pairs of points in the orbit, and one consisting of pairs with a point in each orbit. The latter, that there is only one orbital consisting of such pairs, may be checked using GAP. This means that the automorphism group of a graph admitting all permutations in $\A$ as automorphisms is $S_6\oplus S_6$, not $\A$.
 
In case of $A_n^{(k)}$, $k \geq 2$,  the situation is similar. The difference is that in this case there are two orbitals corresponding to every two  orbits of $A_n^{(k)}$ consisting of pairs with a point in each orbit. It can be easily seen, without GAP, that the orbitals of $A_n^{(k)}$ are the same as those of  $S_n^{(k)}$, and the former is not the automorphism group of any graph.  
\end{proof}


\begin{thebibliography}{99}
 
\bibitem{AC} M. Albertson and K. Collins, Symmetry breaking in graphs, Electron. J. Combin. 3 (1996), R18.
 
\bibitem{alb} M. Albertson, Distinguishing cartesian powers of graphs, Electron. J. Combin. 12 (2005), R17.

 
\bibitem{BC} R. F. Bailey, P. J. Cameron, Base size, metric dimension and other invariants of groups and graphs, Bull. Lond. Math. Soc. 43 (2011), 209-242.
 
 
\bibitem{cha} M. Chan, The maximum distinguishing number of a group, Electron. J. Combin. 13 (2006), R70.
 
\bibitem{cha2} M. Chan, The distinguishing number of the direct product and wreath product actions, J. Algebraic Comb. 24 (2006), 331–345.
 

 
 
 
 
\bibitem{CL} P. J. Cameron, J. H. van Lint, ``Designs, graphs, codes and their links'', Cambridge University Press, Cambridge, 1991.
 
 
 
\bibitem{CNS} P. J. Cameron, P. M. Neumann, J. Saxl, On groups with no regular orbits on the set of
subsets, Arch. Math. 43 (1984), 295-296.
 
\bibitem{CTr} K. L. Collins and A. N. Trenk, The distinguishing chromatic number, Electron. J. Combin. 13 (2006) R16.
 
\bibitem{CT} M. Conder, T. Tucker, Motion and distinguishing number two, Ars Math. Contemp. 4 (2011), 63–72.
 
 
\bibitem{dev} A. Devillers, L. Morgan, S. Harper, The distinguishing number of quasiprimitive and semiprimitive groups, Arch. Math. 113 (2019), 127-139.
 
\bibitem{DM} J. D. Dixon, B. Mortimer, Permutation Groups, Springer-Verlag Press, New York, 1996.
 
 
\bibitem{glu} D. Gluck, Trivial set-stabilizers in finite permutation groups, Canad. J. Math. 35 (1983), 59-67.
 
 
\bibitem{GK1} M. Grech, A. Kisielewicz, {Symmetry groups of Boolean functions}, European J. Combin. 40 (2014), 1-10.
 
 
\bibitem{IK}  M. Imrich and S. Klav\v zar, Distinguishing cartesian powers of graphs, J. Graph Theory 53 (2006), 250–260.
 
 
 
 
 
 
\bibitem{kis1} A. Kisielewicz,
{Symmetry groups of boolean functions and constructions of permutation groups,}
J. Algebra {199} (1998), 379-403.
 
\bibitem{KWZ} S. Klav\v zar, T.-L. Wong and X. Zhu, Distinguishing labellings of group action on vector spaces and graphs, J. Algebra 303 (2006), 626–641.


\bibitem{PT} M. Pil\'sniak, T. Tucker, Distinguishing index of maps, European J. Combinatorics 84 (2020) 103034.
 
\bibitem{rs} A. Russell, R. Sundaram, A note on the asymptotics and computational complexity of graph distinguishability, Electron. J. Combin. 5 (1998), R23.
 
 
 
\bibitem{ser} A. Seress, Primitive groups with no regular orbits on the set of subsets, Bull. Lond. Math. Soc. 29 (1997), 697-704.
 
 
\bibitem{SV} J. Siemons, F. D. Volta, {Orbit equivalence and permutation groups defined by unordered relations}, J. Algebraic Combin. 35 (2012), 547-564.
 
\bibitem{tuc} T. Tucker, Distinguishing maps, Electron. J. Combin. 18 (2011) P50.
 
\bibitem{tuc2} 
T. Tucker, Distinguishing maps II: the general case, Electron. J. Combin. 20 (2013) P50.
 
\bibitem{tym} J. Tymoczko, Distinguishing Numbers for Graphs and Groups, Electronic J. Combin. 11 (2004), R63.
 
 
\end{thebibliography}
\end{document}